\documentclass[12pt]{article}

\usepackage[resetlabels,labeled]{multibib}
\usepackage[utf8]{inputenc}
\usepackage{amsfonts}
\usepackage[english]{babel}
\usepackage{mathrsfs}
\usepackage{amssymb}
\usepackage{amsthm}
\usepackage{amsmath}
\usepackage{tikz-cd}
\usepackage{caption}
\usepackage{fancyhdr}
\usepackage{scalerel,stackengine}
\stackMath
\usepackage{lipsum,hyperref}

\usetikzlibrary{arrows}
\newtheorem{theorem}{Theorem}
\numberwithin{theorem}{section}
\newtheorem*{theorem*}{Theorem}
\newtheorem{definition}[theorem]{Definition}
\newtheorem*{definition*}{Definition}
\newtheorem{lemma}[theorem]{Lemma}

\newtheorem{proposition}[theorem]{Proposition}
\newtheorem*{Proposition*}{Proposition}
\newtheorem*{remark}{Remark}
\newtheorem*{lemma*}{Lemma}
\newcommand{\R}{\mathbb{R}}
\newcommand{\calS}{\mathcal{S}}
\newcommand{\scrC}{\mathscr{C}}
\newcommand{\scrF}{\mathscr{F}}
\newcommand{\calR}{\mathcal{R}}
\newcommand{\bpartial}{\overline{\partial}}
\newcommand{\changefont}{%
    \fontsize{10}{11}\selectfont
}

\title{On the holomorphic convexity of reductive Galois coverings over compact K\"ahler surfaces}
\author{Yuan Liu}
\date{}

\fancypagestyle{specialfooter}{%
  \fancyhf{}
  \fancyhead[]{}

  \fancyfoot[L]{
  \vspace{2pt}
  \changefont \textit{Date}: \today
  \\
  }
}

\begin{document}
\maketitle
\thispagestyle{specialfooter}
\begin{abstract}
This article generalizes the result of Katzarkov and Ramachandran from algebraic surfaces to K\"ahler surfaces. We follow their argument to prove the holomorphic convexity of a reductive Galois covering over a compact K\"ahler surface which does not have two ends, except that we replace the $p$-adic factorization theorem by an analysis of the singularities of the continuous subanalytic plurisubharmonic exhaustion function.
\end{abstract}
\section{Introduction} \label{introduction}
The Shafarevich conjecture asks the following question: ``Is the universal covering of a compact K\"ahler manifold holomorphically convex?" A complex manifold is holomorphically convex if for any sequence of discrete points, we can always find a holomorphic function which is unbounded there. A relatively easier question would be: ``Can we construct a nonconstant holomorphic function on the universal covering of a compact K\"ahler manifold?" If we can show that some intermediate Galois coverings are holomorphically convex, then there are non-constant holomorphic functions on the universal covering, which are pullback of holomorphic functions on these intermediate coverings.\\
For the following, we only focus on complex surfaces where more tools are available. Let $X$ be a compact K\"ahler surface, $X'$ be a Galois covering over $X$ with group $\Gamma$ and $\rho:\Gamma\to GL_N(\mathbb{C})$ be a faithful discrete reductive representation. It is standard to construct a $\Gamma$-equivariant harmonic map, say $u$, from $X'$ into a naturally associated non-positively curved (NPC) space $N(\rho)$. By Bochner-Siu-Sampson formula, it is pluriharmonic. Given that the reductive representation is discrete, we can compose with the distance function squared on the target to get a weakly plurisubharmonic exhaustion function $f$ with non-empty degeneracy locus of its Levi form. Even though the existence of a weakly plurisubharmonic exhaustion function does not always imply holomorphic convexity, we can expect this for $X'$ by using some additional properties. If the exhaustion function is generically (i.e. away from a proper subvariety) strictly plurisubharmonic, by a theorem of Narasimhan (see Lemma 3.1 of \cite{KR} or theorem \ref{deg} below), we can show the holomorphic convexity of $X'$. Otherwise, the Levi form of the function $f$ degenerates everywhere. Now we have (i) a holomorphic foliation induced by the pluriharmonic map, and (ii) a semi-K\"ahler $(1,1)$-form working as a lower bound of Levi form of $f$ and the level set of $f$ is Levi-flat and foliated by this holomorphic foliation. The possible target of $u$ to consider is either a Riemannian locally symmetric space of noncompact type, a locally compact Euclidean building, or a direct product of these targets and flat Euclidean factors. When the target does not contain any Euclidean building factor, the function $f$ is real analytic and the result of Napier and Ramachandran (Theorem 4.6, \cite{NR1995}) implies holomorphic convexity. We call this the \textit{real analytic or smooth case} in the following. When there is a nontrivial Euclidean building factor in the target, the function $f$ is continuous subanalytic, and we will call this the \textit{subanalytic or non-smooth case} correspondingly.\\ 
\\
This paper aims to generalize the result of \cite{KR} from smooth projective surfaces to K\"ahler surfaces. In the non-smooth case, their proof relies on the $p$-adic factorization theorem (c.f. Theorem 1.3 \cite{KR}), which is only available on smooth projective varieties. Now we give a different proof using subanalyticity instead and thus generalize their result to K\"ahler surfaces. The scheme of proof is as following. Assume that we have a continuous subanalytic plurisubharmonic exhaustion function $f$, with its Levi-flat level set foliated by a complex codimension $1$ holomorphic foliation $\scrF$, then there are two situations: either there is one regular real analytic level set to proceed as in the smooth case, or we have singularities of $f$ (see definition \ref{critical_regular} below) lying in each generic level set. In the latter case, the degeneracy locus of the Levi form in each level set of $f$ is given by leaves of $\scrF$. Because $f$ is constant on the leaves, it degenerates locally to a subanalytic function defined on a disc. This shows the set of singularities of $f$ is also holomorphically foliated and we can use the subanalyticity and a Sard type theorem for subanalytic functions (see theorem \ref{sard}) to show the intersection of the set of singularities of $f$ with its generic level set coincides with a finite union of holomorphic leaves there. This process shows that we have enough many closed leaves, which are compact since they lie in the level set of an exhaustion function. Then we can use a standard argument on connectedness to show the existence of a proper holomorphic map from the reductive K\"ahler covering onto a Riemann surface with connected fibers. In particular, this shows $X'$ is holomorphically convex. The key theorem in this subanalytic case can be stated as:
\begin{theorem*}[\textbf{Key theorem in subanalytic case}]
Let $M$ be a connected noncompact complete K\"ahler manifold which does not have two ends and such that for each compact subset $K$ of $M$, there is an automorphism $\gamma$ of $M$ with $\gamma(K)\cap K=\emptyset$. If there is a continuous subanalytic plurisubharmonic exhaustion function $f$ and its generic Levi-flat level sets are holomorphically foliated by a given holomorphic foliation $\scrF$ of complex codimension $1$. Then $M$ admits a proper holomorphic map onto a Riemann surface with connected fibers. In particular, $M$ is holomorphically convex.
\end{theorem*}
This can be seen as the subanalytic version of Theorem 4.6 in \cite{NR1995}. Even though we only prove it for K\"ahler surfaces (see theorem \ref{convexity_subanalytic}), the argument there goes to arbitrary dimension almost verbatim. In fact, we restrict to K\"ahler surfaces simply to get a plurisubharmonic exhaustion function satisfying all the above conditions, which comes from a harmonic map associated with a reductive representation (see lemma \ref{non-archimedean}). With the key theorem working as a replacement of the $p$-adic factorization theorem in the subanalytic case, we prove a K\"ahler version of the main theorem in \cite{KR}. Furthermore, this process will be helpful when we consider more general targets, such as hyperbolic buildings or DM-complexes defined by Daskalopoulos and Mese \cite{DM}.\\
\\
\textbf{Acknowledgement:} The author would like to express sincere gratitude to Professor Mohan Ramachandran, for his advice, guidance, and constant encouragement on the research.

\section{The strategy of using harmonic maps} \label{harmonic_map}
The existence of a harmonic map is closely related to the following concept.
\begin{definition} \label{reductive}
We say a finitely generated group $\Gamma$ is reductive if it admits a faithful representation $\rho$ into a complex Lie group $G$, and its monodromy group (i.e. the Zariski closure of its image $\rho(\Gamma)$) is a reductive subgroup of $G$. 
\end{definition}
Hereafter, we will call a connected Galois covering $X'$ over a compact K\"ahler surface $X$ with a linear reductive infinite Galois group $\Gamma$ simply as a reductive covering surface $X'$ (over $X$ with group $\Gamma$). Also, all the representations below are assumed to be faithful unless stated otherwise.\\
If $G$ is defined over the archimedean field (i.e. $\mathbb{R}$ or $\mathbb{C}$), the intrinsically associated non-positively curved space $N(\rho)=G\slash K$ (here $K$ is a maximal compact subgroup of $G$) is a Riemannian locally symmetric space of non-compact type and $G$ acts isometrically on $G\slash K$ by left action. If $G$ is defined over a non-archimedean locally compact local field (i.e. a finite extension $E_{\mathfrak{p}}$ of $\mathbb{Q}_{\mathfrak{p}}$ where $\mathfrak{p}$ is a non-archimedean place), the intrinsically associated non-positively curved space $N(\rho)$ is a locally compact Euclidean building and $G$ acts on it by isometries.
\begin{definition}
A map $u:X'\to N(\rho)$ is called $\Gamma$-equivariant if for any $\gamma \in \Gamma$ and $x\in X'$, we have
$$u(\gamma\cdot x)= \rho(\gamma)\cdot u(x),$$
where $\Gamma$ acts on $X'$ by the deck transformation and $\rho(\Gamma)\subseteq G$ acts on $N(\rho)$ as isometries.
\end{definition}
Next, we talk about how to use a $\Gamma$-equivariant harmonic map to construct a real analytic or a continuous subanalytic plurisubharmonic exhaustion function with its level sets holomorphically foliated in the smooth case or non-smooth case separately. Additionally, we show how to get holomorphic convexity in the smooth case. The results in this section are well known and we include them for completeness.
\subsection{Real analytic (smooth) case}
Now the target $G\slash K$ is a Riemannian locally symmetric space of non-compact type. Because it is topologically contractible, we can construct a nonconstant continuous $\Gamma$-equivariant map $v$ (For example, see proof of Lemma 2.1, \cite{LR}). Since $X=X'/\Gamma$ is compact, the map $v$ is of finite $\Gamma$-equivariant energy (defined on $X$).\\
Now we need the following special case of the theorem by Corlette (c.f. corollary 3.5, \cite{Corlette}).
\begin{theorem}
Let $X'$ be a reductive covering surface and $G$ be a semisimple Lie group without compact factors. Let $K$ be a maximal compact subgroup of $G$ so that $G\slash K$ carries the structure of a Riemannian locally symmetric space of noncompact type. Then there is a nonconstant $\Gamma$-equivariant harmonic map $u:\ X'\to G\slash K$ with finite energy.
\end{theorem}
If $G$ is defined over $\mathbb{C}$, by Weil's restriction of scalars, we can see $G(\mathbb{C})$ as the set of $\R$-points of $\text{Res}_{\mathbb{C}|\R}(G)$. Now we may assume that $G$ is a real reductive Lie group, and the target is a real Riemannian locally symmetric space $G\slash K$ of noncompact type and has strongly negative curvature in the complexified sense. From the Bochner-Siu-Sampson formula, we have $u$ is pluriharmonic in the sense that its restriction to any complex analytic curve on $X'$ is harmonic (c.f. Theorem 2, \cite{Mok}).\\
We can define a holomorphic structure on the smooth complex vector bundle  $u^{*}(T_N^{\mathbb{C}})$, with respect to which the one-form $\theta:= \partial u$ is a holomorphic section of $T^{*(1,0)}_{X'}\otimes u^{*}(T_N^{\mathbb{C}})$ on $X'$. The distribution given by the kernel of $\theta$ defines a holomorphic foliation $\scrF$ on the complement of a complex analytic subvariety $V$ of complex codimension $\geqslant 2$ (i.e. the set of indeterminacy of $\theta$). In our situation, $V$ is a set of discrete points. Aslo, $u$ is constant on each leaf of $\mathscr{F}$ (For details, refer to \cite{Mok}, \cite{Siu2} or \cite{Eys1}).\\
Define 
$$f(x)= d^2\big(u(x),y_0\big)$$
where $y_0$ is a fixed point on the target and $d$ is the distance function induced by the Riemannian structure there. Then $d^2$ is strictly convex and real analytic. If we assume that $\Gamma$ is discrete, then $f(x)$ is a real analytic plurisubharmonic exhaustion function on $X'$. Also, it is constant on the leaves of the holomorphic foliation $\scrF$ defined by the holomorphic 1-form $\theta=\partial u$.\\
We summarize the above argument as:
\begin{lemma}[\textbf{Real analytic case}]\label{archimedean}
Let $X'$ be a reductive covering surface with group $\Gamma$ and the representation $\rho: \Gamma\to G$ be faithful and discrete. Assume $G$ is a real Lie group defined over the archimedean field, then there is a real analytic plurisubharmonic exhaustion function $f$ on $X'$ which is constant on the leaves of a holomorphic foliation $\scrF$ defined by a (closed) holomorphic 1-form $\theta=\partial u$, where $u$ is the associated $\Gamma$-equivariant harmonic map from $X'$ to $N(\rho)$.
\end{lemma}
If $f$ is generically strictly plurisubharmonic, then $X'$ is holomorphically convex by a theorem by Narasimhan (For its proof, see lemma 3.1 \cite{KR}). Since any open Riemann surface is Stein, we can rewrite it for surfaces as:
\begin{theorem} \label{deg}
Let $X'$ be a connected noncompact complete complex surface admitting a continuous plurisubharmonic exhaustion function $f$ which is generically strictly plurisubharmonic. Then $X'$ is holomorphically convex.
\end{theorem}
Otherwise, we have $\sqrt{-1}\partial\overline{\partial}f$ degenerates everywhere. For any $x\in X'$, we can choose local coordinates $\{z_1,z_2\}$ such that $z_1$ provides the direction along which $\sqrt{-1}\partial\overline{\partial}f$ vanishes, then $(\sqrt{-1}\partial\overline{\partial}f\wedge \sqrt{-1}\partial\overline{\partial}f)(x)=0$. We can rewrite this as $\sqrt{-1}\partial\overline{\partial}f\wedge \sqrt{-1}\partial\overline{\partial}f\equiv 0$. A direct calculation gives:
$$\sqrt{-1}\partial\overline{\partial}f=(d^2)''\cdot\sqrt{-1}\partial u\wedge\overline{\partial}u+(d^2)'\cdot\sqrt{-1}\partial\overline{\partial}u=(d^2)''\cdot|\partial u|^2$$
because $u$ is pluriharmonic. Since $d^2$ is strictly convex, we have $\sqrt{-1}\partial\overline{\partial}f=0$ if and only if $\partial u=0$, which is exactly given by the leaves of $\scrF$. Since the $1$-dimension complex submanifold in each regular level set is given by the holomorphic leaves of $\scrF$ and $f$ is constant there, the regular level set of $f$ is Levi-flat and the degeneracy locus of its Levi-form there is given by the holomorphic leaves of $\scrF$. Now we are in the place to apply the theorem 4.8 of \cite{NR1995}, and we rewrite it for  K\"ahler surfaces in the real analytic case.
\begin{theorem}\label{analytic}
Let $X'$ be a Galois covering over a compact K\"ahler surface with group $\Gamma$. Assume that $X'$ does not have two ends and admits a real analytic plurisubharmonic exhaustion function with some regular real analytic Levi-flat level set, then $X'$ is holomorphically convex.
\end{theorem}
We include the proof here for the convenience of readers.
\begin{proof}
If $X'$ has at least three ends, then it admits a proper holomorphic map onto a Riemann surface with connected fibers by a classical result of Napier and Ramachandran (c.f. Theorem 3.4 \cite{NR1995}). Now we consider the case when $X'$ has exactly one end. By assumption, we can take a regular value $r$ of $f$ such that the real analytic level set $X'_r:= \{x\in X':\ f(x)=r\}$ is Levi-flat in the sense that $\sqrt{-1}\partial\bpartial{f}\equiv0$ on the maximal $1$-dimension holomorphic tangent subspace contained in $X'_r$. Choose a relatively compact neighborhood $U$ of the given level set $X'_r$ where $f$ is real analytic, and consider $g(x)=e^{-c\cdot f(x)}$, where $c$ is a sufficiently large constant to be determined. With respect to the hermitian metric induced by the background K\"ahler metric on the holomorphic tangent space, we can write $v\in T_x^{(1,0)}(U)$ as $v=v_1+v_2$, where $v_1\in T_x^{(1,0)}\big(X'_{f(x)}\big)$ and $v_2$ is in its orthogonal complement.\\
Denote $\mathcal{L}(g)$ to be the Levi-form of function $g$, then
\begin{equation*}
\begin{split}
\mathcal{L}(g)(v,v)=&e^{-cf}\big(c^2\cdot |\partial f(v)|^2-c\cdot \mathcal{L}(f)(v,v)\big)\\
=&c\cdot e^{-cf}\big(c|\partial f(v_2)|^2-\mathcal{L}(f)(v_2,v_2)\big)
\end{split}
\end{equation*}
Since $U$ is relatively compact, we can assume that $|\partial f(v_2)|^2\geqslant \mu|v_2|^2$ and $\mathcal{L}(f)(v_2,v_2))\leqslant \nu|v_2|^2$ for some $\mu, \nu>0$. Then $\mathcal{L}(g)(v,v)\geqslant c\cdot e^{-c\cdot f}(c\cdot\mu-\nu)|v_2|^2\geqslant 0$ if we choose $c>\nu/\mu$ and is strictly positive if $v_2\neq 0$.
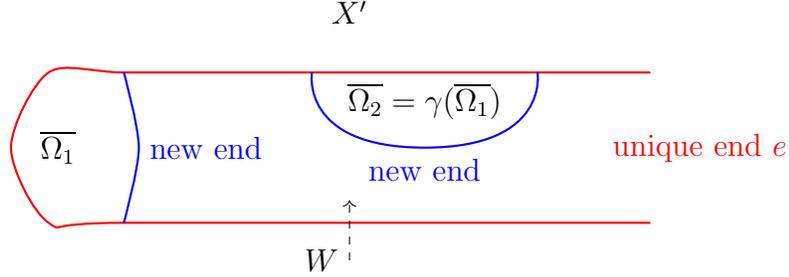
\begin{figure}[ht!]
\centering
\begin{tikzpicture}
\draw [red,thick] plot [smooth] coordinates {(5,1) (2,1) (0,1) (-2,1) (-3,1) (-3.5,0) (-3,-1) (-2,-1) (5,-1)};
\draw[thick,blue] plot [smooth] coordinates {(-2,1) (-1.8,0) (-2,-1)};
\draw[thick,blue] plot [smooth,tension =2] coordinates {(3.5,1) (2,0) (.5,1)};
\node[below] at (2,1) {$\overline{\Omega_2}=\gamma(\overline{\Omega_1})$};
\node[left] at (-2.5,0) {$\overline{\Omega_1}$};
\node[red, text width=3cm] at (6,0) {unique end $e$};
\node[blue, right] at (-1.8,0) {new end};
\node[blue, below] at (2,0) {new end};
\draw[dashed,->,below] (1,-1.5)--(1,-.7);
\node[left] at (1,-1.5) {$W$}; 
\node[above] at (1,1.5) {$X'$}; 
\end{tikzpicture}
\label{fig:my_label}
\end{figure}
\\
This shows that $g(x)=e^{-c\cdot f(x)}$ is a weakly plurisubharmonic function on a neighborhood $U$ of $X'_r$. Denote $\Omega_1$ to be the union of relatively compact sublevel set $\{x\in X': f(x)<r\}$ with all compact components of its complement. By the maximal principle, the function $f$ is constant on these compact components. In particular, we have $f=r$ on $\partial \Omega_1$. Then $X'-\overline{\Omega_1}$ is the unique connected noncompact component (the end). By construction, $g$ is constant on $\partial \Omega_1$, say $g=a$. Then $-\log{(a-g)}$ provides the plurisubharmonic exhaustion function near $\partial \Omega_1$ on $U-\overline{\Omega_1}$. By patching $g$ and $f$ together (for example, take the function $\max{(f,-\log{(a-g)})}$), we get that $X'-\overline{\Omega_1}$ is connected and weakly 1-complete.\\
Since $\Gamma$ is properly discontinuous, we can find $\gamma\in\Gamma$ such that $\Omega_2=\gamma(\Omega_1)$ satisfies $\overline{\Omega_2}\cap\overline{\Omega_1}=\emptyset$. Then $\gamma(X'-\Omega_1)=X'-\gamma(\Omega_1)=X'-\Omega_2$ is connected as the image of a connected set. Also, since $\gamma$ is biholomorphic, $g\circ\gamma^{-1}$ gives a weakly plurisubharmonic function near the boundary of $\Omega_2$. Take $W$ to be the unique connected component of $X'-(\overline{\Omega_1}\cup\overline{\Omega_2})$ which is noncompact. By patching $-\log{(a-g)},-\log{(a-g)}\circ\gamma^{-1}$ and $f$ as before, we get a weakly plurisubharmonic exhaustion function on $W$. In conclusion, we show that $W$ is a connected weakly 1-complete K\"{a}hler surface with three ends.\\
Now we have a proper holomorphic map $\tau$ from $W$ onto a Riemann surface $R$ with connected fibers. By analytic continuation, we can extend $\tau$ to $X'$. Precisely, choose $U$ to be an open set in the regular values of $\tau$ in $R$ and two distinct points $p$ and $q$ there. Then the connected K\"{a}hler surface $V:=X'-\tau^{-1}(\{p,q\})$ is weakly 1-complete with three ends. Thus there is a connected Riemann surface $S$ and a proper holomorphic map $\sigma: V\to S$ with connected fibers. The maps $\tau$ and $\sigma$ together give a holomorphic map from $X'$ to some Riemann surface $T$ by identifying each point $x$ in $U-\{p,q\}$ with $\sigma(\tau^{-1}(x))$ in $S$. In particular, $X'$ is holomorphically convex.\\
\end{proof}
Combine the results of lemma \ref{archimedean}, theorem \ref{deg} and theorem \ref{analytic}, we get the following well-known result in the real analytic case:
\begin{theorem}[\textbf{Holomorphic convexity in real analytic case}]\label{archimedean_convexity}
Let $X'$ be a reductive covering surface which does not have two ends and the representation defined over the archimedean field be faithful and discrete, then $X'$ is holomorphically convex.
\end{theorem}
\subsection{Subanalytic (non-smooth) case}
Here, the target $Y=N(\rho)$ is a locally compact Euclidean building. It is still contractible and we can similarly construct a nonconstant $\Gamma$-equivariant map $v$ into $Y$ with finite energy, for $X$ is compact.\\
The existence of a nonconstant $\Gamma$-equivariant harmonic map of finite energy now is guaranteed by the theory of Gromov and Schoen (c.f. Theorem 7.1, \cite{GS} for locally compact F-connected complexes). For our special situation, we have:
\begin{theorem}
Let $X'$ be a reductive covering surface and $G$ be a reductive complex Lie group defined over a non-archimedean locally compact local field. Denote $Y=N(\rho)$ to be the associated locally compact Euclidean building. If there is a non-constant Lipschitz $\Gamma$-equivariant map $v:X'\to Y$ with finite energy, then there is a $\Gamma$-equivariant Lipschitz harmonic map $u$ of least energy, in the sense that the restriction of $u$ to a small ball about any point is energy minimizing.
\end{theorem}
Let us recall the definition of the singular and regular sets for harmonic maps into locally compact F-connected complexes.
\begin{definition} \label{regular_singular_harmonic}
Let $u$ be a harmonic map from a Riemannian manifold $X$ to a locally compact F-connected complex $Y$. A point $x\in X$ is called a regular point if there is $\sigma>0$ such that $u(B_{\sigma}(x))$ is contained in one F-flat (in the case of buildings, we use the terminology ``apartment'' instead). Otherwise, it is called a singular point. Denote the set of regular points and that of singular points as $\calR(u)$ and $\calS(u)$ respectively.
\end{definition}
Also, the definition of a pluriharmonic map is modified as:
\begin{definition}
We say that a harmonic map $u$ from a K\"{a}hler manifold to a locally compact F-connected complex is pluriharmonic if it is pluriharmonic in the usual sense on its regular set $\calR(u)$.
\end{definition}
Although we can only do the smooth differential geometry calculation on $\calR(u)$, the set $\calS(u)$ has Hausdorff codimension at least $2$ and the Bochner-Siu-Sampson formula still works (c.f. Theorem 7.3, \cite{GS}). The result is:
\begin{theorem}
A finite energy equivariant harmonic map from a K\"{a}hler manifold into a locally compact F-connected complex is pluriharmonic.
\end{theorem}
Now we need the subanalyticity of the distance function squared on a locally compact Euclidean building. Before that, we recall the definitions of semianalytic  and subanalytic sets and functions (c.f. \cite{BM}).
\begin{definition}
Let $M$ be a real analytic manifold and $X$ be a subset of $M$.
\begin{itemize} 
\item[(a)] The subset $X$ is semianalytic if each point of $M$ has a neighbourhood $U$ such that $X\cap U$ is defined by equations and inequalities given by real analytic functions on $U$.
\item[(b)] The subset $X$ of $M$ is subanalytic if each point of $M$ admits a neighbourhood $U$ such that $X\cap U$ is a projection of a relatively compact semianalytic set (i.e. there is a real analytic manifold $N$ and a relatively compact semianalytic subset $A$ of $M\times N$ such that $X\cap U =\pi(A)$, where $\pi:M\times N\to M$ is the projection).
\item[(c)] A function $f:X\to \R$ is semianalytic(subanalytic) if its graph is semianalytic(subanalytic) in $M\times \R$.
\end{itemize}
\end{definition}

\begin{lemma}
Let $y_0$ be a fixed point on a locally compact Euclidean building $Y$ of dimension $k$, then the square of distance function defined as $$d^2(y):=d^2(y,y_0)$$
is subanalytic.
\end{lemma}
\begin{proof}
Assume the strata containing $y_0$ with the lowest dimension is $Z_0$ and write the union of all apartments containing $Z_0$ as $\bigcup\limits_{i=1}^{l}Y_l$. On each $Y_i$, it is a Euclidean space of dimension $k$ and the distance function squared $d_i^2(y):= d^2(y, y_0)|_{Y_i}$ is real analytic. Also, being real analytic or subanalytic is a local property and we only need to consider the intersection with a small neighborhood of $y_0$.\\
Now denote the chamber whose closure contains $Z_0$ as $C_0$ (i.e. $C_0$ has $Z_0$ as frontier or $C_0=Z_0$). If $y_0$ is an interior point of $C_0$, then the distance function squared is locally that on $C_0$, which is real analytic. If $y_0$ is a boundary point of $C_0$, then $Y_i-\overline{C_0}$
(for any $2\leqslant i\leqslant l$) is a semianalytic subset and the restriction of distance function squared there is semianalytic. Take the coordinates $(x_i,t_i)\in Y_i\times\R$ in the graph of $d^2|_{Y_i}$ and define $M= Y_1\times (Y_2-\overline{C_0})\cdots\times (Y_l-\overline{C_0})$, $N=\{(t_1,\cdots,t_l): t_1=\cdots=t_l\}$(i.e. the diagonal which is real analytic) and $\pi:M\times N\to M\times\R$ is the natural projection given by $\pi(x_1,\cdots,x_l;t_1,\cdots,t_l)=(x_1,\cdots,x_l,t_1)$, then the graph of function $d^2(y)$ is given by $\pi(M\times N)$ and is subanalytic.  By definition, $d^2(y)$ is a continuous subanalytic function.\\
\end{proof}
\begin{remark}
By the same argument, we can prove that the distance function squared on an NPC DM-complex is subanalytic \cite{DM}.
\end{remark}
In addition, $d^2$ is strictly convex because the target is non-positively curved (in the sense of \cite{GS}). If we assume the action of $\Gamma$ is discrete, by composing with the pluriharmonic map $u$, we get that 
$$f(x)=d^2\big(u(x),y_0\big)$$
is a continuous subanalytic plurisubharmonic exhaustion function on $X'$. From the work of Eyssidieux (c.f. Proposition 3.3.6 of \cite{Eys1}), there is a closed real semipositive (1,1)-form $\omega_{\rho}$ (i.e. a semi-K\"ahler (1,1)-form) on the spectral covering $s: \widehat{X'}\to X'$. It can be written as $$\omega_{\rho}=s^{*}\sqrt{-1}\sum\limits_{i=1}^N \lambda_i\wedge\bar{\lambda_i}=:s^*(\omega_\rho')$$
where $\lambda_i$'s are holomorphic 1-forms well-defined on $X'$. Notice that these $\lambda_i$'s differ from each other by a Weyl group action, thus their kernels coincide and define a holomorphic foliation $\scrF$ on $\calR(u)$. Also, since the singular set $\calS(u)$ is complex analytic of codimension 1 and $\lambda_i$ is bounded for $u$ is Lipschitz in the interior, we can extend $\scrF$ across $\calS(u)$, which is well defined by holomorphic 1-form $\theta$ on the spectral covering $\widehat{X'}$. Then the function $f$ is constant on the leaves lying in $\calR(u)$ as in the smooth case. This is still true on $\calS(u)$ because we can approach the leaves inside it by leaves in the regular set $\calR(u)$, where $f$ is a constant.
\begin{remark}
For the definition of spectral coverings, see section 3.4 in \cite{Eys2}. For a more constructive way without using spectral coverings, see section 2.2.2 in \cite{Klingler}.
\end{remark}
If $f$ is generically strictly plurisubharmonic, we get the holomorphic convexity from theorem \ref{deg}. If not, since (lemma 3.3.12, \cite{Eys1})
$$\sqrt{-1}\partial\overline{\partial}f\geqslant c\cdot \omega_\rho'$$
in the sense of currents, we see that the degeneracy locus of the Levi-form of $f$ is given by these holomorphic leaves. The above argument is summarized as the following lemma.
\begin{lemma}[\textbf{Subanalytic case}] \label{non-archimedean}
Let $X'$ be a reductive covering surface with group $\Gamma$ and $\rho: \Gamma\to G$ be a faithful discrete representation into a complex Lie group defined over a non-archimedean locally compact local field, then there is a holomorphic foliation $\scrF$ given by a holomorphic 1-form $\theta$ being well defined on the spectral covering $\widehat{X'}$ of $X'$. Also, there is a continuous subanalytic plurisubharmonic exhaustion function $f$ on $X'$. It is constant on the leaves of $\scrF$ and its degeneracy locus of Levi-form is given by these holomorphic leaves.
\end{lemma}
We may expect to get an analogous result on holomorphic convexity as in the smooth case. This will be done in the next section.
\section{Proof of holomorphic convexity in subanalytic case}\label{proof}
Now, we want to get the holomorphic convexity in the subanalytic case. From the result of Lemma \ref{non-archimedean}, we only need to prove the following.
\begin{theorem}\label{convexity_subanalytic}
Let $X'$ be a reductive covering surface which does not have two ends and $\scrF$ be a holomorphic foliation of codimension 1 defined by a (closed) holomorphic 1-form $\theta$. If there is a subanalytic plurisubharmonic exhaustion function $f$ with its generic Levi-flat level set holomorphically foliated by the leaves of $\scrF$, then $X'$ admits a proper holomorphic map onto a Riemann surface with connected fibers. In particular, $X'$ is holomorphically convex.
\end{theorem}
Set $M=X'$. The indeterminacy set $\mathcal{I}$ of the holomorphic 1-form $\theta$ is of complex codimension 2, which is a set of discrete points when $M$ is a complex surface. Then $f(\mathcal{I})$ is also a set of discrete values in $\R$, because $f$ is proper. For the following choices of level sets of $f$, we always avoid this subset.\\
The proof can be divided into the following two cases:\\
\\
\textit{Case I: there are enough many real analytic level sets of $f$.}\\
In fact, we only need to find one real number $r\notin f(\mathcal{I})$ such that $f^{-1}(r)$ is real analytic. Then we can imitate the proof of theorem \ref{analytic} to get a proper holomorphic map from $M$ onto a Riemann surface with connected fibers.\\
\\
\textit{Case II: Almost all level sets of $f$ are not real analytic.}\\
This simply means that we cannot find $r\notin f(\mathcal{I})$ such that $f^{-1}(r)$ is real analytic. This case can be seen as the subanalytic version of Theorem \ref{analytic}.\\
First, we define \textit{critical/regular values and singular/regular sets for subanalytic functions}.
\begin{definition}\label{critical_regular}
Let $N$ be a subanalytic subset of a real analytic manifold $M$, and $f$ be a real-valued continuous subanalytic function on $N$, we say that $r\in f(N)$ is a critical value of $f$ if $f^{-1}(r)$ is \textit{not} a real analytic submanifold of $M$. Otherwise, it is called a regular value\footnote{\ Here we use the convention that the number not lying in the image set (i.e. $\forall r\notin f(N)$) is a regular value.}. Denote the set of critical values as $\mathscr{C}(f)$.\\
We say that $x\in N$ is a regular point if $f$ is real analytic at $x$; otherwise, we say that $x$ is a singular point. The set of regular and singular points are called regular set and singular set of $f$ respectively.\footnote{\ This definition of regular/singular set for subanalytic functions is different from that for harmonic maps in Definition \ref{regular_singular_harmonic}.}
\end{definition}
The existence of compact leaves is the key step to getting a proper holomorphic map onto a Riemann surface with connected fibers. Usually, we would try to show by comparison of dimensions that a leaf coincides with a level set of one holomorphic function or a level set of two linearly independent real-valued functions, which are naturally closed. But for here, we have only one real-valued function $f$, and the real dimension of its level set is 1 higher than that of the leaves. In this situation, we need to use the fact that $f$ is subanalytic to analyze its singular set lying in a fixed level set (which is also closed as the intersection of two closed subsets and is holomorphically foliated), which has the correct dimension and gets a chance to coincide with the leaves inside it. The subanalytic structure is essential here.\\
The most difficult part is stated as:
\begin{theorem}\label{allsingular}
Let $M$ be a Galois covering over a compact K\"ahler surface and $f$ be a continuous subanalytic plurisubharmonic exhaustion function which is constant on the leaves of a given codimension 1 holomorphic foliation $\scrF$. If every level set of $f$ is \textit{not} real analytic, then for almost every $r\in f(M)$, the set of singular points lying in the level set $f^{-1}(r)$ is given by a finite union of compact leaves of $\scrF$.
\end{theorem}
\begin{proof}
The proof of this theorem is divided into following steps.\\
\textit{Step (i): Propositions of subanalytic functions/sets.}\\
We list here all the properties of subanalytic functions/sets needed (c.f. the work by Bierstone and Milman \cite{BM} for proof).
\begin{proposition}[Theorem 0.1, \cite{BM}: Uniformization theorem] \label{uniform}
Suppose $M$ is a real analytic manifold and $X$ is a closed subanalytic subset of $M$. Then there is a real analytic manifold $N$ of the same dimension as $X$ and a proper real analytic map $\phi: N\to M$ such that $\phi(N)=X$.
\end{proposition}
\begin{proposition}[Proposition 5.3, \cite{BM}: Resolution of subanalytic functions] \label{resolution}
Suppose $M$ is a real analytic manifold and $f$ is a continuous subanalytic function on $M$. Then there is a real analytic manifold $N$, of the same dimension as $M$, and a proper surjective real analytic map $\phi: N\to M$ such that $f\circ\phi$ is real analytic on $N$.
\end{proposition}
\begin{proposition}[Theorem 6.1, \cite{BM}: low dimensional subanalytic sets] \label{lowdim}
Let $M$ be a real analytic manifold and let $X$ be a subanalytic subset of $M$. If $\ \dim(M)\leqslant 2$ or $\dim(X)\leqslant 1$, $X$ is semianalytic.
\end{proposition}
\begin{proposition}[Theorem 3.10, \cite{BM}: theorem of the complement] \label{complement}
Let $M$ be a real analytic manifold and let $X$ be a subanalytic subset of $M$. Then $M-X$ is subanalytic.
\end{proposition}
\begin{proposition}[Theorem 7.5, \cite{BM}: subanalyticity of the regular set] \label{realanalytic}
Let $N$ be a real analytic manifold and let $f: N\to \R$ be a continuous subanalytic function. Then the regular set of $f$ is an open subanalytic subset of $N$.
\end{proposition}
\begin{proposition}[Theorem 3.14, \cite{BM}: Bound on the number of fibers for subanalytic map]\label{bound}
Let $M$ and $N$ be real analytic manifolds and $X$ be a relatively compact subset of $M$. Let $\phi:X\to N$ be a subanalytic map. Then the number of connected components of a fiber $\phi^{-1}(y)$ is bounded locally on $N$.
\end{proposition}
Also, the subanalytic sets are ``nicely'' stratified in the sense that any point admits a locally compact neighborhood which can be written as the union of smooth real analytic manifolds of different dimensions, called strata. The lower dimensional strata are either disjoint from the higher dimensional strata, or are contained in them as the frontier (c.f. section 16, \cite{Loj} for details).\\
\\
\textit{Step (ii): Global picture.} \\
Due to Proposition \ref{resolution}, we have the following resolution of $f$:\\
\[
  \begin{tikzcd}
    N \arrow{r}{\phi} \arrow[swap]{dr}{g} & M \arrow{d}{f}\\
     & \mathbb{R}
  \end{tikzcd}
\]
where $N$ is a 4-dimensional real analytic manifold, $\phi$ is a real analytic map of generically rank 4 onto $M$ and $g=f\circ \phi$ is a real analytic function on $N$.\\
Denote $E\subseteq M$ to be the singular set of $f$. By Proposition \ref{realanalytic}, $M-E$ is open and subanalytic. Then $E$ is a closed subanalytic subset of $M$ by Proposition \ref{complement}. Let $D$ be the set on $N$ where $\phi$ has a lower rank, then $D$ is of real dimension at most 3. On the complement of $\phi(D)$, we can write locally that $f=g\circ \phi^{-1}$, which shows that $f$ is real analytic. Thus we show that $E\subseteq\phi(D)$ is of real dimension at most 3.\\
\\
\textit{Step (iii): Local picture.}\\
Locally, every $x\in M$ admits a relatively compact open subset $V$ of $M$ which is compatible with the foliation $\scrF$. Write the coordinate on $V\cong\Delta\times\Delta$ as $\{z_1,z_2\}$, where $\Delta$ denotes the unit disc on complex plane. The transversal direction to the leaves is given by $z_1$ and the leaf direction is given by $z_2$. Since $f$ is constant on the leaves, locally $f=f(z_1)$ descends to a function $\tilde{f}(z_1)$ on $\Delta$. In conclusion, we have the following diagram:\\
\[
  \begin{tikzcd}
     \Delta\times\Delta \arrow{r}{\pi_1} \arrow[swap]{d}{f} & \Delta \arrow{dl}{\tilde{f}}\\
      \mathbb{R} &
  \end{tikzcd}
\]
where $\pi_1$ denotes the projection onto the first factor.\\
Denote $\tilde{E}=\{z_1\in\Delta:\ \tilde{f}\text{\ is not real analytic at\ }z_1\}$ to be the singular set of $\tilde{f}$, then locally 
$$E\cap V=\tilde{E}\times\Delta.$$
Since $\tilde{E}$ is a 1-dimensional subanalytic subset of $\Delta$, it is semianalytic by Proposition \ref{lowdim}. Then $E$ can be seen as a disc bundle over the semianalytic set $\tilde{E}$.\\
In conclusion, we have the following commutative diagram:\\
\[
  \begin{tikzcd}
    N \arrow{r}{\phi} \arrow[swap]{ddr}{g} & M  \arrow{dd}{f}  & &  V\cong\Delta\times\Delta \arrow[left hook->,swap]{ll}{\text{inclusion}}\arrow{r}{\pi_1}\arrow[swap,dashed]{lldd}{(f)}& \Delta\arrow{ddlll}{\tilde{f}}\\
    \\
     & \mathbb{R}
  \end{tikzcd}
\]
\\
\textit{Step (iv): Stratification.}\\
We now stratify $\tilde{E}$ into 1-dimensional and 0-dimensional real analytic manifolds, i.e. $$\tilde{E}=\big(\bigcup\limits_{i} \tilde{E}_i^{0}\big)\bigcup\big(\bigcup\limits_{j}\tilde{E}_j^{1}\big)$$
where $\tilde{E}_i^{0}$'s are the 0-strata (points) and $\tilde{E}_j^{1}$'s are the 1-strata (curves). Here and following, we use an upper index to show the dimension of strata.\\
This local stratification gives the corresponding part of $E$ as\\
$$E=\big(\bigcup\limits_{i} E_i^{2}\big)\bigcup\big(\bigcup\limits_{j}E_j^{3}\big)$$
where the 2-strata $E_i^{2}$ is the leaf going through $\tilde{E}_i^{0}$ and the 3-strata $E_j^{3}$ is locally $\tilde{E}_j^{1}\times\Delta$ and globally the union of all leaves going through points in $\tilde{E}_j^{1}$. We can summarize this by saying that the singular set $E$ is holomorphically foliated by $\scrF$.\\
If we have one 2-strata, it is a leaf of $\scrF$ lying in the connected components of the level set of $f$ with the same real dimension $2$, thus they coincide. This shows that all 2-strata are compact leaves because $f$ is exhaustive. Otherwise, if almost all the strata are 3-dimensional, we will use the following.\\
\\
\textit{Step (v): Sard-type theorem on 3-strata}.\\
From now on, we focus on the restriction of $f$ to its singular set $E$, denoted as 
$$\widehat{f}=f\big|_{E}.$$
Then $\widehat{f}$ is a continuous subanalytic function defined on a 3-dimensional closed subanalytic set $E$.\\
Due to Proposition \ref{uniform}, we have a 3-dimensional real analytic manifold $F$ and a proper real analytic map $\psi_1: F\to E$ such that $\psi_1(F)=E$. Then $k:=\widehat{f}\circ \psi_1$ is a subanalytic function on a real analytic manifold $F$. By applying Proposition \ref{resolution}, the resolution of subanalytic functions, to $k$, we get the following diagram:\\
\[
  \begin{tikzcd}
    G \arrow[rr,bend left,"\psi"]\arrow{r}{\psi_2}\arrow{drr}[swap]{h} & F \arrow{r}{\psi_1}\arrow{dr}{k}& E\arrow{d}{\widehat{f}}\\
     & & \R
  \end{tikzcd}
\]
Here $G$ is a 3-dimensional real analytic manifold, $h$ is a real analytic function and $\psi_2$ is a proper real analytic map onto $F$. Then $\psi:=\psi_1\circ\psi_2$ is a proper real analytic map from a 3-dimensional real analytic manifold $G$ onto $E$.\\
Denote the set of critical values of the real analytic function $h$ in the usual sense as $\scrC(h)$. For any $r\in \R-\scrC(h)$, $h^{-1}(r)$ is a 2-dimensional real analytic manifold. Then $\dim\psi\big(h^{-1}(r)\big)\leqslant 2$. Since $\widehat{f}$ is constant on leaves, its level set $\widehat{f}^{-1}(r)$ is of real dimension at least 2. Also, $\widehat{f}^{-1}(r)=\psi\big(h^{-1}(r)\big)$ as an easy consequence of the following lemma.
\begin{lemma*}
If $l=n\circ m: X\xrightarrow{m} Y \xrightarrow{n} Z$ and $m$ is surjective, then $n^{-1}(z)=m\big(l^{-1}(z)\big)$, for any  $z\in Z$.
\end{lemma*}
\begin{proof}
One direction is obvious: $m\big(l^{-1}(z)\big)\subseteq n^{-1}(z)$. Conversely, for any $y\in n^{-1}(z)$, there is $x\in X$ with $m(x)=y$, then $z=n(y)=n\big(m(x)\big)=l(x)$. So $x\in l^{-1}(z)$. This shows that $n^{-1}(z)\subseteq f\big(l^{-1}(z)\big)$ and thus we have the equality.
\end{proof}
This forces $\widehat{f}^{-1}(r)$ to be of real dimension 2 and its connected component to coincide with one leaf of $\scrF$ (This in fact shows that $\widehat{f}^{-1}(r)$ is not only real analytic but also complex analytic). Then $r$ is a regular value for $\widehat{f}$. This shows that $\scrC(\widehat{f})\subseteq \scrC(h)$ is of measure zero thanks to the classical Sard's theorem for the real analytic function $h$.\\ 
Then the Sard-type theorem for the subanalytic function $\widehat{f}$ is proved. We can rewrite it for a slightly more general situation without any difficulty.
\begin{theorem} [Sard-type theorem for special subanalytic functions]\label{sard}
Let $N$ be a closed subanalytic subset of a real analytic manifold $M$ and $g$ be a subanalytic function defined on $N$. Assume that there is a real analytic foliation $\mathscr{G}$ of real codimension 1 in $N$ such that the level set of $g$ is foliated by $\mathscr{G}$ with the exception at most on a subset of measure zero in $g(N)\subseteq\R$, then the set of critical values of $g$ is of measure zero in $\R$.
\end{theorem}
Now we can finish the proof. For any $r\notin \scrC(\widehat{f})\cup f(\mathcal{I})$ (a set of measure zero), we have that 
$$\widehat{f}^{-1}(r)=f^{-1}(r)\cap E\neq\emptyset,$$
for no level set of $f$ is real analytic.\\
Furthermore, the singular set of $f$ lying in $f^{-1}(r)$ (which is closed) is of real codimension 2 and thus coincides with a discrete union of some holomorphic leaves of $\scrF$ contained in $f^{-1}(r)$. The closed leaf is compact because it is contained in the level set of an exhaustion function $f$ and the discrete union contains finitely many compact leaves. This shows the existence of at least one compact leaf in each generic level set of $f$.
\end{proof}
The remaining step from theorem \ref{allsingular} to prove theorem \ref{convexity_subanalytic} is standard. Denote $Q$  to be the set of points in $M$ which lie in some compact leaves of $\scrF$. Then $Q$ is nonempty. We will use the argument on connectedness below.\\
Firstly, we show that $Q$ is open following the work by Napier and Ramachandran (c.f. proof of Proposition 2.3, \cite{NR1997}). Let $x_0\in Q$ and $L_0$ be the compact leaf it lies in. It can be globally defined as $\theta=dF$ for some holomorphic function $F: U\to\mathbb{C}$ in a relatively compact neighborhood $U$ of $L_0$ in $M-\Big(F^{-1}\big(F(x_0)\big)- L_0\Big)$. Since $\partial U$ is compact, we can find a neighborhood $V\subseteq \mathbb{C}$ of $F(x_0)$ such that $F^{-1}(V)\cap \partial U=\emptyset$. Hence the leaf given by $F^{-1}(w)$ for any $w\in V$ is also compact. Thus $x_0\in F^{-1}(V)\subseteq Q$ and $Q$ is open. \\
Next, we show that $Q$ is closed by considering the volume of these compact leaves. The volume of closed subvarieties of dimension $k$ is a homotopy invariant, for its volume form $d \text{Vol}=\frac{\omega^k}{k!}$ is closed where $\omega$ is the K\"ahler form on $M$. Assume that $x_n\to x_0$ when $n\to\infty$ and $L_n$ is the compact leaf containing $x_n$. Denote $L_0$ to be the leaf going through $x_0$. The volumes of $L_n$'s are bounded by $V$ and $L_n$'s are of bounded recurrence $N$ (c.f. Proposition \ref{bound}) in the sense of Definition 3 of \cite{Mok}. Then $L_0$ is also of finite volume, which is bounded by $N\cdot V$. Because $M$, as a Galois covering over a compact K\"ahler surface, is of bounded geometry, we must have $L_0$ is also compact. This shows that $x_0\in Q$ and thus $Q$ is closed.
\begin{remark}
We can prove the closeness slightly easier in our case. With the same notation, we have $L_n\subset \{x\in M:f(x)=f(x_n)\}$, then the limit $f(x_0)=\lim\limits_{n\to\infty}f(x_n)$ exists and $L_0\subset \{x\in M:f(x)=f(x_0)\}$. Since $f$ is exhuastive, we know that $L$ is compact.
\end{remark}
Now, since $Q$ is nonempty, open, and closed and $M$ is connected, we must have $Q=M$. Denote the holomorphic equivalent relation $\mathscr{R}$ given by these compact leaves (i.e. $x\sim y$ if and only if they lie in the same compact leaf), then we get a Riemann surface $R=M/\mathscr{R}$ and the natural holomorphic map $\alpha: M \to R$ is proper and has connected fibers by construction. In particular, $M$ is holomorphically convex. This finished the proof of the holomorphic convexity in the subanalytic case.\\
\\
With Theorem \ref{convexity_subanalytic} working as a replacement of the $p$-adic factorization theorem, we get the K\"ahler version of Katzarkov and Ramachandran's result on algebraic surfaces (c.f. Main Theorem, \cite{KR}) in the next section.
\section{Generalization from algebraic surfaces to K\"ahler surfaces} \label{applications}
In this section, we will provide a different proof of the main theorem (Theorem 1.2) of \cite{KR}. The proof there uses the $p$-adic factorization theorem (c.f. Theorem 1.3 \cite{KR}) to deal with the non-smooth case, and this is only available when $X$ is smooth projective. With the above theorem \ref{convexity_subanalytic} at hand, we can generalize their result from smooth projective (algebraic) surfaces to K\"ahler surfaces.
\begin{theorem}\label{shaf}
If $X'\to X$ is a reductive Galois covering over a compact K\"{a}hler surface with a faithful representation $\rho$ and it does not have two ends, then $X'$ is holomorphically convex.
\end{theorem}
Even though the essential adjustment is only in the non-smooth case, we include all the details from \cite{KR} here for the convenience of readers and point out explicitly where we apply the theorem \ref{convexity_subanalytic} instead of the $p$-adic factorization theorem.\\
Firstly, we can deform the reductive representation $\rho:\Gamma\to G$ to a discrete representation $\rho'$ as follows. By taking a finite covering of $X'$, we may assume that group $G$ is torsion-free. Write $G= S\times C$, where $C$ is the maximal compact abelian factor contained in $G$ and $S$ is the semisimple part of $G$.\\ 
Denote $\rho_C: \Gamma\to G\to C=G/S$. By Lemma 2.1 of \cite{KR}, every abelian representation $\rho_C$ can be replaced with another representation $\rho_C': \Gamma\to C$ which is defined over $\overline{\mathbb{Q}}$ and whose kernel is commensurable to $\text{ker}(\rho_C)$. Similarly, we consider the semisimple representation $\rho_S: \Gamma\to G\to S=G/C$. By Lemma 2.2 of \cite{KR}, we have a neighborhood $W$ of $\rho_S$ where any representation has a Zariski dense image. In particular, there is a Zariski dense representation $\rho_S'$ defined over $\overline{\mathbb{Q}}$ in $W$. Now consider the construction of adelic type for $\rho_S'$. Since $\Gamma$ is finitely generated, $\rho_S'$ is defined over a finite extension $E$ of $\mathbb{Q}$ and we consider the completion with distinct valuations. When the valuation place is non-archimedean, say $\mathfrak{p}$, we have the completion $E_{\mathfrak{p}}$ with action on a locally compact Euclidean building $B_{\mathfrak{p}}$; when the valuation place is archimedean, say $\nu$, we have the completion $E_\nu$ with action on a Riemannian locally symmetric space Symm$_\nu$ of noncompact type. Consider the direct product $Y_1=\prod\limits_{\nu}{\text{Symm}}_{\nu}\times \prod\limits_{\mathfrak{p}} B_{\mathfrak{p}}$ with the action of $\Gamma$ on the diagonal, where only finitely many terms can be nontrivial. Since $E_{\mathfrak{p}}$ is discrete in $\prod\limits_{\nu}\text{E}_{\nu}\times \prod\limits_{\mathfrak{p}} E_{\mathfrak{p}}$, the image of $\rho_S'$ acts discretely on $Y_1$.\\
After doing the deformation on $\rho_C$ and $\rho_S$ separately, we get a discrete representation $\rho'$ of $\Gamma$ over $\overline{\mathbb{Q}}$, acting on $Y=Y_1\times \mathbb{C}^k$, where $Y_1$ comes from the deformation of $\rho_S$ and $\mathbb{C}^k$ comes from the deformation of $\rho_C$. Denote $Y= N(\rho')$, then $N(\rho')$ is a non-positively curved space and we can construct a $\Gamma$-equivariant pluriharmonic map $u: X'\to N(\rho')$, which comes from harmonic maps into each factors as shown in section 1.\\
Recall the definition of the rank of a pluriharmonic map.
\begin{definition}
The rank of a pluriharmonic map is the codimension of the holomorphic foliation it induces.
\end{definition}
Now we are ready to prove Theorem \ref{shaf}. As $X'$ is of complex dimension $2$, the rank of a pluriharmonic map $u$ is either $1$ or $2$.
\begin{proof}[Proof of Theorem \ref{shaf}]
We consider the following two cases.\\
{\it Case 1. there is a representation $\rho'$ defined over $\overline{\mathbb{Q}}$ for which the associated pluriharmonic map $u$ has rank $2$.}\\
This means the dimension of the leaves of the foliation $\scrF$ given by $u$ is generically zero. Let $X''$ be the Galois covering that corresponds to the image of $\rho':\Gamma\to G\leqslant GL_N(\overline{\mathbb{Q}})$, i.e.\\
\[
  \begin{tikzcd}
    X' \arrow{dr}{\ker{(\rho')}}\arrow{dd}[swap]{\Gamma}&\\
     & X''\arrow{ld}{\rho'(\Gamma)}\\
     X&\\
  \end{tikzcd}
\]
Take $u:X''\to N(\rho')$ to be the associated $\Gamma$-equivariant pluriharmonic map and define $f(x)=d^2\big(u(x), y_0\big)$, where $d$ is the distance function on $N(\rho')$ and $y_0$ is a fixed point there. The degeneracy locus of $\sqrt{-1}\partial\bar{\partial}f$ given by the leaves of this foliation is a set $B$ of discrete points and the function $f$ is strictly plurisubharmonic away from this subset. In addition, it is an exhaustion function for $\rho'$ is discrete. By thoerem \ref{deg}, $X''$ is holomorphically convex. Then we may consider its Cartan-Remmert reduction $S(X'')$, which is a complex surface. By the result of Napier \cite{Napier}, the obstruction of holomorphic convexity for this surface $X''$ is the existence of an infinite chain 
(i.e. a connected noncompact analytic curve all of whose irreducible components are compact). Assume $Z''$ is a finite chain of compact irreducible components in $X''$ whose lift to $X'$ is an infinite chain $Z'$, then $\ker{\rho'}$ has to be infinite.\\
Now we are in the place to apply the result by Lasell and Ramachandran (c.f. section 2, \cite{LR}), which is stated for K\"ahler surfaces as follows:
\begin{lemma}
Let $X'$ be a reductive Galois covering surface over a compact K\"ahler surface $X$ with group $\Gamma$ and its representation $\rho$, $Z'=\bigcup\limits_{i=1}^{\infty} D_i$ be a connected analytic subvariety of $X'$, all the irreducible components $D_i$'s of which are compact. If $\rho(\pi_1(D_i))$ is trivial for any $i$, we have $\rho$ restricted to $\pi_1(Z')$ factors through a finite group $\Delta_n$, the order of which depends only on the rank $n$ of $G$.
\end{lemma}
Because the image of $\rho:\pi_1 (D_i)\to G$ is finite for each $i$, by the above lemma, we have $\rho(\pi_1(Z'))$ is also finite. Since $\rho$ is faithful, this means $\pi_1(Z')$ is finite, which contradicts the fact that $Z'$ is an infinite chain.\\
We summarize the above result as the following.
\begin{lemma} \label{intermediate}
Let $X_2\to X_1\to X$ be a tower of infinite Galois coverings over a compact K\"ahler surface $X$. Assume that the Galois group $\Gamma_2=\text{Galois}(X_2/X)$ admits a(an) (almost) faithful linear reductive representation into a complex Lie group and the intermediate covering $X_1$ is holomorphically convex, then $X_2$ is also holomorphically convex.
\end{lemma}
{\it Case 2: there is a neighborhood $W$ of $\rho$ such that for every representation $\rho'\in W$ defined over $\overline{\mathbb{Q}}$, the rank of the associated pluriharmonic map $u$ is $1$.}\\
As before, we may assume $\rho'$ is a discrete representation with Zariski dense image in a reductive complex Lie group $G$ over $\overline{\mathbb{Q}}$. Since any group with two ends contains an infinite cyclic group $\mathbb{Z}$ of finite index, we know the image of $\rho'$ cannot have two ends.\\
If $N(\rho')$ contains no Euclidean building factor, we apply Theorem \ref{archimedean_convexity}; if there is at least one non-trivial Euclidean building factor in $N(\rho')$ for any $\rho'\in W$, we use Theorem \ref{convexity_subanalytic}. In both cases, we get a holomorphic map from some intermediate covering surface $X''$ onto a Riemann surface with connected fibers, which shows $X''$ is holomorphically convex. Then by Lemma \ref{intermediate}, we know $X'$ is holomorphically convex.
\end{proof}
\bibliographystyle{alpha}
\bibliography{references}

\begin{thebibliography}{Mok92}

\bibitem[BM88]{BM}
E.~Bierstone and P.~Milman.
\newblock Semianalytic and subanalytic sets.
\newblock {\em Publications Math\'ematiques de l'IH\'ES}, 67:5--42, 1988.

\bibitem[Cor88]{Corlette}
K.~Corlette.
\newblock {Flat $G$-bundles with canonical metrics}.
\newblock {\em Journal of Differential Geometry}, 28(3):361 -- 382, 1988.

\bibitem[DM16]{DM}
G.~Daskalopoulos and C.~Mese.
\newblock {\em On the singular set of harmonic maps into DM-complexes}.
\newblock Memoirs of the American Mathematical Society, 2016.

\bibitem[Eys04]{Eys1}
P.~Eyssidieux.
\newblock Sur la convexité holomorphe des revêtements linéaires réductifs
  d’une variété projective algébrique complexe.
\newblock {\em Inventiones mathematicae}, 156(3):503–564, 2004.

\bibitem[Eys11]{Eys2}
Philippe Eyssidieux.
\newblock Lectures on the shafarevich conjecture on uniformization.
\newblock {\em Complex manifolds, foliations and uniformization}, 34:35, 2011.

\bibitem[GS92]{GS}
M.~Gromov and R.~Schoen.
\newblock Harmonic maps into singular spaces and $p$-adic superrigidity for
  lattices in groups of rank one.
\newblock {\em Publications Math\'ematiques de l'IH\'ES}, 76:165--246, 1992.

\bibitem[Kli03]{Klingler}
Bruno Klingler.
\newblock Sur la rigidité de certains groupes fondamentaux,
  l’arithméticité des réseaux hyperboliques complexes, et les “faux
  plans projectifs”.
\newblock {\em Inventiones mathematicae}, 153(1):105 -- 143, 2003.

\bibitem[KR98]{KR}
L.~Katzarkov and M.~Ramachandran.
\newblock On the universal coverings of algebraic surfaces.
\newblock {\em Annales Scientifiques de l’École Normale Supérieure}, Vol.
  31(4):525--535, 1998.

\bibitem[{\L}oj64]{Loj}
S.~{\L}ojasiewicz.
\newblock {\em Ensembles semi-analytiques}.
\newblock Inst. Hautes Études Sci., Bures-sur-Yvette, 1964.

\bibitem[LR96]{LR}
B.~Lasell and M.~Ramachandran.
\newblock Observations on harmonic maps and singular varieties.
\newblock {\em Annales scientifiques de l'\'Ecole Normale Sup\'erieure}, Ser.
  4, 29(2):135--148, 1996.

\bibitem[Mok92]{Mok}
N.~Mok.
\newblock Factorization of semisimple discrete representations of {K}ähler
  groups.
\newblock {\em Inventiones mathematicae}, 110(3):557--614, 1992.

\bibitem[Nap90]{Napier}
T.~Napier.
\newblock Convexity properties of coverings of smooth projective varieties.
\newblock {\em Mathematische Annalen}, 286(1-3):433--480, 1990.

\bibitem[NR95]{NR1995}
T.~Napier and M.~Ramachandran.
\newblock Structure theorems for complete {K}ähler manifolds and applications
  to {L}efschetz type theorems.
\newblock {\em Geometric and functional analysis}, 5(5):809--852, 1995.

\bibitem[NR97]{NR1997}
T.~Napier and M.~Ramachandran.
\newblock The {B}ochner-{H}artogs dichotomy for weakly 1-complete {K}\"ahler
  manifolds.
\newblock {\em Annales de l'Institut Fourier}, 47(5):1345--1365, 1997.

\bibitem[Siu87]{Siu2}
Y.T. Siu.
\newblock Strong rigidity for k{\"a}hler manifolds and the construction of
  bounded holomorphic functions.
\newblock In {\em Discrete Groups in Geometry and Analysis: Papers in Honor of
  G.D. Mostow on His Sixtieth Birthday}, pages 124--151, Boston, MA, 1987.
  Birkh{\"a}user Boston.

\end{thebibliography}

\end{document}